\documentclass[12pt, twoside]{article}
\usepackage{amsmath,amsxtra,amssymb,latexsym, amscd,amsthm,pb-diagram,fancyhdr,euscript}
\usepackage{amsthm}
\usepackage{latexsym}
\usepackage{amssymb}
\usepackage{amsmath}
\usepackage{indentfirst}
\usepackage{hyperref}
\usepackage{mathrsfs}
\usepackage{array}
\usepackage{euscript}
\usepackage{slashed}
\usepackage[T1]{fontenc}
\usepackage{graphicx}
\usepackage[utf8]{inputenc}
\usepackage[french,english]{babel}
\usepackage{multicol}
\hypersetup{
   colorlinks,
    citecolor=blue,
    filecolor=black,
    linkcolor=blue,
    urlcolor=magenta
}
\hypersetup{linktocpage}

\newtheorem{theorem}{Theorem}[section]
\newtheorem{lemma}[theorem]{Lemma}

\newtheorem{definition}[theorem]{Definition}

\newtheorem{remark}[theorem]{Remark}

\newtheorem{assumption}[theorem]{Assumption}

\newcommand{\norm}[1]{\left\Vert#1\right\Vert}

\numberwithin{equation}{section}
\usepackage{amsfonts}

\usepackage{amsmath}
\usepackage{amssymb}

\def\dive{\mathrm{div}}

\def\geq{\geqslant} 
\def\leq{\leqslant}

\def\r{\hro{R}}

\def\hro{\mathbb}

\def\C+{C_+([t_0,\infty))}

\begin{document}\mbox{}
\vspace{0.25in}

\begin{center}

{\huge{\bf Well-posedness and stability for a class of solutions of semi-linear diffusion equations with rough coefficients}}

\vspace{0.25in}

\large{ {\bf Le The Sac}\footnote{Faculty of Computer Science and Engineering, Thuyloi University, 175 Tay Son, Dong Da, Hanoi, Vietnam. Enail: SacLT@tlu.edu.vn} and
{\bf Pham Truong Xuan}\footnote{Thang Long Institute of Mathematics and Applied Sciences (TIMAS), Thang Long University, Nghiem Xuan Yem, Hanoi, Vietnam. Email: phamtruongxuan.k5@gmail.com or xuanpt@thanglong.edu.vn} 
} 

\end{center}

\begin{abstract}  
In this work we study the existence, uniqueness and polynomial stability of the pseudo almost periodic mild solutions of semi-linear diffusion equations with rough coefficients in certain interpolation spaces. First, we rewrite the equations in abstract parabolic equation. Then, we use the polynomial stability of the semigroups of the corresponding linear equations to prove the boundedness of the solution operator for the linear equations in appropriate interpolation spaces. We show that this operator preserves the pseudo almost periodic property of functions. We will use the fixed point argument to obtain the existence and stability of the pseudo almost periodic mild solutions for the semi-linear equations. The abstract results will be applied to the semi-linear diffusion equations with rough coefficients to obtain our desired results. 
\end{abstract}

\tableofcontents
\section{Introduction and Preliminaries}
\subsection{Introduction}
The study of mild solutions of difference and differential equations has been the center of studies of many mathematicians. Especially, topics related to the existence, uniqueness and asymptotic behaviours of periodic, almost periodic, pseudo almost periodic mild solutions and their generalizations. These solutions and their properties have significant applications in many areas such as physics, mathematical biology, control theory, and others (see for examples \cite{Ch,DuHu,HZha}).
Historically, the notion of pseudo almost periodic functions was introduced initially by Zhang (see \cite{Zh1,Zh2}). Then, intensive studies of this concept of solution and its generalizations to differential  and difference equations have been made during recent years (see for examples \cite{Di2006,Di2007,EzFaGu,ZL} and the references therein). All of these works consider the evolution equations where the corresponding semigroups are both exponential stable. 

In the present paper, we will investigate the existence, uniqueness and stability of the pseudo almost periodic mild solutions to 
the semi-linear diffusion equations with rough coefficients
\begin{equation}\label{Rou11}
u'(t) -  b\Delta u(t) = g(t,u), \, \, \, (t,x)\in \mathbb{R}\times \mathbb{R}^d,
\end{equation}
where $b : \mathbb{R}^d \to \mathbb{C}$ is a  measurable function satisfying $b\in L^\infty(\mathbb{R}^d)$ and $\mathrm{Re}\,\, b \geqslant \delta > 0$ for some $\delta > 0$, $g(t,u) = |u(t)|^{m-1}u(t) + F(t)$ for some fixed $m \in \mathbb{N}$ and $F$ is a given bounded function on $\mathbb{R}$.

We know that the operator $-A$ defined on $L^p(\mathbb{R}^d)$ by $Au := -b\Delta u$ generates a bounded analytic semigroup (also called ultracontractive semigroup) $T(t):= e^{-tA}$ on $L^p(\mathbb{R}^d)$ for all $1<p<\infty$ (for more details see \cite[Section 7.3.2]{Are04}) such that 
$$(T(t)f)(x) = \int_{\mathbb{R}^d}K(t,x,y)f(y)dy, \, \, \, t>0 \hbox{  and a.e  } x,y \in \mathbb{R}^d,$$
where $K(t,x,y)$ is the heat kernel. The semi-linear equation \eqref{Rou11} can be rewritten as
\begin{equation}\label{Rou22}
u'(t) + Au(t) = g(t,u), \, \, \, (t,x)\in \mathbb{R}\times\mathbb{R}^d.
\end{equation}
The corresponding linear equation is
\begin{equation}\label{LRou1}
u'(t) + Au(t) = F(t), \,\,\, (t,x) \in \mathbb{R}\times \mathbb{R}^d.
\end{equation}
In general, we consider these problems  on the interpolation spaces to a large class of semi-linear evolution equations of the form
\begin{equation}\label{abstract}
u'(t)+Au(t)=BG(u)(t), \, t\in \mathbb{R},
\end{equation}
where $-A$ is the generator of a $C_0$-semigroup $(e^{-tA})_{t\geq 0}$ on some interpolation spaces and $B$ plays the role of a ''connection'' operator between the various spaces involved. Note that, we have $B=\mathrm{Id}$ in the case of equation \eqref{Rou11}.

One of the important features in our strategy is Assumption \ref{A} on the polynomial estimates of the operator $e^{-tA}B\,\, (t>0)$. Equations of type \eqref{abstract} associated with $e^{-tA}$ and $B$ satisfying these estimates occur in many situations such as the equations of fluid dynamic equations and various diffusion equations with rough coefficients (see \cite{GeHiHu,HiHuySe,HuyHaXuan}).

The novelty and difficultly in our study appear from the fact that we allow the zero number to belong to the spectrum $\sigma(A)$. This leads to the problem that the semigroup $(e^{-tA})_{t\geq 0}$ is no longer exponential stable. However, the polynomial estimates of $e^{-tA}B\, (t>0)$ are still sufficiently good to allow us to handle the corresponding linear equation
\begin{equation}\label{linearE}
u'(t)+Au(t)=Bf(t), \, t\in \mathbb{R},
\end{equation}
where $f$ is a pseudo almost periodic (PAP-) function (see Definition \ref{WPAAfunction} for the notion of PAP-functions). Using the polynomial estimates in Assumption \ref{A} of $e^{-tA}B\,\, (t>0)$, we can
construct the interpolation spaces and then apply the interpolation theorem to obtain the
boundedness of the solution operator on the spaces of PAP-functions.
Namely, we will prove that if $f$ is a PAP-function, then the corresponding
mild solution $u(t)$ of \eqref{linearE} is also PAP-function (see Theorem \ref{thm1}).
Then, we use fixed point argument to extend this result to the semi-linear equation \eqref{abstract} under an assumption that the Nemytskii operator $G$ is a locally Lipschitz continuous operator that
maps PAP-functions into PAP-functions (see Assumption \ref{GAA}). Consequently, we obtain the existence and uniqueness of the pseudo almost periodic mild solution in the PAP-space to \eqref{abstract} (see Theorem \ref{thm2}). Moreover, the interpolation spaces also allow us to prove the polynomial stability of such pseudo almost periodic mild solutions under Assumption \ref{assum} of the operators $e^{-tA},\, B$ and Nemytskii operator $G$ (see Theorem \ref{Stabilitytheorem}). Finally, we apply our abstract results to  the semi-linear diffusion equations with rough coefficients (see Theorem \ref{app}).

This paper is organized as follows: Section \ref{Section 2} contains the setting of linear parabolic equations in interpolation spaces and the results of the existence and uniqueness of pseudo
almost periodic mild solutions to these equations; in Section \ref{Section 3} we investigate the semi-linear equations: the existence, the uniqueness and the stability of pseudo
almost periodic mild solutions. Lastly, Section \ref{Section 4} give the application of abstract results to heat equations with rough coefficients \eqref{Rou11}.\\
{\bf Notation.}\\
$\bullet$ We denote the norm on Banach space $X$ by $\norm{.}_{X}$ and the supremum norm on the Banach space $BC(\r,X)$ by $\norm{.}_{\infty,X}$.\\
$\bullet$ Let $Y$ be an interpolation space we use the notation $BC(\mathbb{R}, Y)$ for the space of all time weak-continuous functions $f \in L^\infty(\mathbb{R},Y)$.
\subsection{Preliminaries}
We recall the notions of almost periodic (AP-), pseudo almost periodic (PAP-) functions
as follows:
\begin{definition}\label{1}(AP-function)
For a Banach space $X$, a continuous function $f: \mathbb{R} \to X$ is called (Bohr) almost periodic if for each $\varepsilon > 0$ there exists $l(\varepsilon) > 0$ such
that every interval of length $l(\varepsilon)$ contains a number $T$ with the property that
\begin{equation*}
\norm{f(t+T)-f(t)}_X <\varepsilon \hbox{  for each  } t\in \mathbb{R}.
\end{equation*}
The number $T$ above is called an $\varepsilon$-translation number of $f$.
We denote the set of all almost periodic functions $f:\mathbb{R} \to X$ by $AP(\mathbb{R}, X).$
\end{definition}
Note that $(AP(\r,X), \norm{.}_{\infty,X})$ is a Banach space, where $\norm{.}_{\infty,X}$ is the supremum norm (see \cite[Theorem 3.36]{Dia}). The properties of the almost periodic functions can be found in \cite{AH,Kos,Li}.
\begin{definition}\label{WPAAfunction}(PAP-function)
A continuous function $f: \mathbb{R} \rightarrow X$ is called pseudo almost periodic if it can be decomposed as $f = g + \phi$ where $g \in AP(\mathbb{R},X)$ and $\phi$ is a bounded continuous function with vanishing mean value i.e
$$\lim_{L\to \infty}\frac{1}{2L}\int_{-L}^L \norm{\phi(t)}_X dt =0.$$
We denote the set of all functions with vanishing mean value by $PAP_0(\mathbb{R},X)$
and the set of all the pseudo almost periodic (PAP-) functions by $PAP(\mathbb{R},X)$. 
\end{definition}
Note that $(PAP(\mathbb{R},X),\norm{.}_{\infty,X})$ is a Banach space, where $\norm{.}_{\infty,X}$ is the supremum norm (see \cite[Theorem 5.9]{Dia}).

We notice that the notion of pseudo almost periodic function is a generalisation of the almost periodic and asymptotically almost periodic ($AAP$-) functions (see the definition of $AAP$- functions in \cite[Section 3.3]{Dia}). Moreover, we have also the extension of this notion to the weighted pseudo almost periodic and weighted pseudo almost automorphic functions (see for examples \cite{ChaZhaoLiliu,Di2008,EzFaGu} and the references therein). We refer the reader to the books \cite{Di07} for more details about $PAP$-functions and $PAP$-spaces.

We recall the notion of Lorentz spaces as follows (see for example \cite{Ya}): for $1 < p < \infty$ and $1 \leq q \leq \infty$, we define the Lorentz space on $\Omega\subseteq \mathbb{R}^d$ as
$$L^{p,q}(\Omega) = \left\{ u \in L^1_{loc}(\Omega): \norm{u}_{p,q} <\infty \right\},$$
where
$$\norm{u}_{p,q} = \left( \int_0^\infty \left(s\mu \left(\left\{x\in \Omega: |u(x)|>s  \right\}\right)^{1/p}\right)^q \frac{ds}{s}  \right)^{1/q} \hbox{   for   } 1\leq q <\infty$$
and
$$\norm{u}_{p,\infty} = \sup_{s>0}s\mu(\left\{x\in \Omega: |u(x)|>s  \right\})^{1/p}.$$
Denote also by
$$L^{p,\infty}: = L^p_\omega \hbox{  called  weak-}L^p \hbox{  space}.$$
Using real-interpolation functor $(.,.)_{\theta,q}$ we have 
$$(L^{p_0}(\Omega),L^{p_1}(\Omega))_{\theta,q} = L^{p,q}(\Omega), \hbox{  where  } 1<p_0<p<p_1<\infty$$
and $0<\theta<1$ such that 
$$\frac{1}{p} = \frac{1-\theta}{p_0} + \frac{\theta}{p_1},\,\, 1\leq q \leq \infty.$$

We recall the following general interpolation theorem (see \cite[Theorem 3.11.2]{BeLo}):
\begin{theorem}\label{GIT} (General interpolation theorem) Let $(X_0, X_1)$ and $(Y_0, Y_1)$ be interpolation couples of quasinormed spaces. Let $T$ be defined on $X_0 + X_1$ such that $T : X_0 \to Y_0$ as well as $T : X_1 \to Y_1$ are
sublinear with quasi-norm $M_0$ and $M_1$, respectively. Then for any $\theta \in (0,\, 1)$ and $q \in [1,\infty]$ it holds that
$$T : (X_0, X_1)_{\theta,q} \to (Y_0, Y_1)_{\theta,q}$$
is sublinear with quasi-norm $M$ bounded by
$$M \leq M_0^{1-\theta}M_1^\theta.$$
\end{theorem}
For the convenience we refer the reader to the books \cite{BeLo,Lu,Tri} which provide the detailed description of the theory for interpolation spaces and the differential operators.
 
\section{The linear parabolic equations}\label{Section 2}
\subsection{The setting of equations on interpolation spaces}\label{2.2}
Let $X,\, Y_1$ and $Y_2$ be Banach spaces, $(Y_1, Y_2)$ be a couple of Banach spaces and $Y:= (Y_1, Y_2)_{\theta,\infty}$ be a real interpolation space for some $0< \theta <1$. We now consider the inhomogeneous linear evolution equation of the form
\begin{equation}\label{NSE_linear}
u'(t) + Au(t) = Bf(t),t \in \mathbb{R},
\end{equation}
where the unknown $u(t)\in Y $, the operator $-A$ is a generator of a $C_0$-semigroup $e^{-tA}$ on $Y_1$ and $Y_2$, $f$ is a function from $\mathbb{R}$ to $X$ and $B$ is the ''connection'' operator between $X$ and $Y$ such that $e^{-tA}B\in\mathcal{L}(X, Y_i)$ for $i=1,2$ and $t\geqslant 0.$

Since $e^{-tA}$ is $C_0$-semigroup on $Y_1$ and $Y_2$, we have that $e^{-tA}$ is also $C_0$-semigroup on $Y=(Y_1,Y_2)_{\theta,\infty}$. This is because that the real interpolation functor $(.,.)_{\theta,\infty}\, (0<\theta<1)$ satisfies the dense condition, i.e, $Y_1\cap Y_2$ is dense in $Y$ (for more detail see \cite[Proposition 3.8 (c)]{Tr19} and \cite[Subsection 1.6.2]{Tri}).

Many practical problems can be written in the form \eqref{NSE_linear}. For examples, the linear form of the diffusion equations with rough coefficients, where $B=\mathrm{Id}$, i.e, the identity operator (see Section \ref{Section 4}), the linear equations of fluid dynamic flows, where $B=\mathbb{P}\dive$ (see \cite{GeHiHu}) and the control theory, where $B$ is the input operator of the system (see \cite{BD,PS} and the references therein).

We assume that $e^{-tA}B$ satisfies the following polynomial estimates:
\begin{assumption}\label{A}
Assume that $Y_i$ has a Banach pre-dual $Z_i$ for $i=1, 2$ (that means $Y_i=Z'_i$) such that ${Z_1} \cap {Z_2}$ is dense in $Z_i$. Let $-A$ be the generator of a $C_0$-semigroup $e^{-t A}$ on $Y_1$ and $Y_2$. Furthermore, suppose that there exist constants $\alpha_1, \alpha_2\in\mathbb{R}$ with $0<\alpha_2<1<\alpha_1$ and $L>0$ such that
\begin{equation}\label{inequalities_basic}
\begin{gathered}
  {\left\| {{e^{ - tA}}Bv} \right\|_{{Y_1}}} \leqslant {L}{t^{ - {\alpha _1}}}{\left\| v \right\|_X}\ ,\ t > 0, \hfill \\
  {\left\| {{e^{ - tA}}Bv} \right\|_{{Y_2}}} \leqslant {L}{t^{ - {\alpha _2}}}{\left\| v \right\|_X} \ ,  \ t > 0, \hfill \\ 
\end{gathered}
\end{equation}
where the operator $B$ is given in the linearized equation \eqref{NSE_linear}.
\end{assumption}
\begin{remark}
\item[(i)] The condition that ${Z_1} \cap {Z_2}$ is dense in $Z_i \,\, (i=1,2)$ guarantees the dual and pre-dual equalities (see Definition \ref{dual}).
\item[(ii)] From Assumption \ref{A} we can see that the spectrum $\sigma(A)$ contains $0$ and the semigroup $e^{-tA}$ is polynomially stable. This is a new point comparing with the previous works (see for examples \cite{Di2006,Di2007}) on PAP-functions. In particular, the semigroups considered in the previous works are both exponentially stable. An example of the equation associated with the semigroup $e^{-tA}$ validating Assumption \ref{A} with $B=\mathrm{Id}$ (the identity operator) can be found in Section \ref{Section 4} of this paper as well as many other examples of the fluid dynamic equations with $B=\mathbb{P}\mathrm{div}$ (see for example \cite{GeHiHu,HiHuySe,HuyHaXuan}).
\end{remark}
\begin{definition}\label{dual}
Since $(Z_1,Z_2)_{\theta,1}$ is a predual space of $Y=(Y_1,Y_2)_{\theta,\infty}$, we say that $B'$ and $e^{-tA'}$ are the dual operators of $B$ and $e^{-tA}$ respectively if for each $\psi\in (Z_1,Z_2)_{\theta,1}$, $f\in X$ and $g\in Y$ we have
$$\left< Bf,\psi\right> = \left< f, B'\psi\right>,$$
$$\left< e^{-tA}g,\psi\right> = \left< g,e^{-tA'}\psi \right>,$$
where $\left<.,.\right>$ is the scalar product between the dual spaces.
\end{definition}

We recall the definition of mild solutions of the equation \eqref{NSE_linear} on the whole time-line axis (see \cite[equation $(7)$]{HiHuySe}):
\begin{definition}
Suppose that the function $(-\infty,t]\ni \tau \mapsto \left< e^{-(t-\tau)A}Bf(\tau),\psi \right> \in \mathbb{C}$ is integrable for each $\psi\in (Z_1,Z_2)_{\theta,1}$. A continuous function $u: \r \to Y$ is called a mild solution to $(\ref{NSE_linear})$ if $u$ satisfies the integral equation
\begin{equation}\label{MILD}
u(t) = \int_{-\infty }^t {{e^{-(t - \tau)A}}Bf(\tau)d\tau } ,t \in \mathbb{R},
\end{equation}
in the weak sense, i.e, for each $\psi \in (Z_1,Z_2)_{\theta,1}$ we have that
\begin{equation*}
\left<u(t) , \psi \right> = \int_{-\infty }^t \left< e^{-(t - \tau)A}Bf(\tau)  , \psi \right> d\tau  ,t \in \mathbb{R}.
\end{equation*}
\end{definition} 
We notice that if $-A$ generates a bounded analytic semigroup $(e^{-tA})_{t\geq 0}$ on $Y$, then the mild solution given by \eqref{MILD} is also a classical solution (see \cite[Proposition 3.1.16]{AH}).

\subsection{The existence and uniqueness of pseudo almost periodic mild solutions to linear equations}
We state and prove the basic estimate in the following lemma:
\begin{lemma}\label{basic_ine}
Suppose that Assumption \ref{A} holds. 
Let $B'$ and $e^{-tA'}$ be the dual operators of $B$ and $e^{-tA}$ respectively, $X'$ be the dual space of $X$. For all $\psi \in (Z_1,Z_2)_{\theta,1}$, the following assertion holds: the function $t\mapsto \left\|B'e^{-tA'}\psi\right\|_{X'}$  belongs to $L^{1}(0;\infty)$ and 
\begin{equation}\label{L1es}
\int_0^\infty\left\|B'e^{-tA'}\psi\right\|_{X'}dt \leqslant \tilde{L}\|\psi\|_{(Z_1,Z_2)_{\theta,1}}
\end{equation}
for a positive constant $\tilde{L}$ which does not depend on $\psi$. Here, the spaces $X$, $Z_1$ and $Z_2$ are given in Definition \ref{A}.
\end{lemma}
\begin{proof}
Since Assumption \ref{A}, we have the following inequalites on the dual spaces
\begin{equation}\label{inequalities_basic'}
\begin{gathered}
  \norm{ B'{e^{ - tA'}}\psi }_{X'} \leqslant {L}t^{ - {\alpha_1}}\norm{ \psi }_{Y_1'}\ ,\ t > 0, \hfill \\
   \norm{ B'{e^{ - tA'}}\psi }_{X'} \leqslant {L}t^{ - {\alpha _2}}\norm{ \psi }_{Y_2'}\ ,  \ t > 0, \hfill \\ 
\end{gathered}
\end{equation}
where $Y_1'$ and $Y_2'$ are the dual spaces of $Y_1$ and $Y_2$ respectively.

For $\psi\in Z_j, \, j=1,2 $ we denote by
\begin{equation}\label{v}
v_{\psi}(t):= \norm{B'e^{-tA'}\psi}_{X'}.
\end{equation}
Using inequalities \ref{inequalities_basic'} and the fact that the canonical embedding $Z_j \to Y_j'$ is an isometry, we get
$$v_{\psi}(t) \leqslant Ct^{-\alpha_j} \norm{\psi}_{Z_j} \hbox{  for  } \psi \in Z_j, \, \, j=1,2.$$
Therefore $v_{\psi} \in L^{1/\alpha_j, \infty}(0,\infty)$ and
\begin{equation}\label{inqualities_basic''}
\norm{v_{\psi}}_{L^{1/\alpha_j, \infty}(0,\infty)} \leq C_j \norm{\psi}_{Z_j} \hbox{  for  } \psi \in Z_j, \, \, j=1,2.
\end{equation}
These inequalities lead us to define the sublinear operator
\begin{align*}
T: Z_1+Z_2  &\to L^{1/\alpha_1, \infty}(0,\infty) + L^{1/\alpha_2, \infty}(0,\infty)\\
\psi &\mapsto v_\psi,
\end{align*}
where $v_\psi$ is defined by \eqref{v}.

The inequalities \eqref{inqualities_basic''} also yield that the operators
$$T:Z_1 \to L^{1/\alpha_1, \infty}(0,\infty) \hbox{  and  } T:Z_2 \to L^{1/\alpha_2, \infty}(0,\infty)$$
are sublinear. Applying Theorem \ref{GIT}, we obtain that the operator
$$T: (Z_1,Z_2)_{\theta,1} \to \left( L^{1/\alpha_1, \infty}(0,\infty), L^{1/\alpha_2, \infty}(0,\infty) \right)_{\theta,1} = L^1(0,\infty)$$
is also sublinear. It follows that
$$\int_0^\infty\left\|B'e^{-tA'}\psi\right\|_{X'}dt \leqslant \tilde{L}\|\psi\|_{(Z_1,Z_2)_{\theta,1}},$$
where $\tilde{L}>0$ is a constant.
\end{proof}
As a consequence of Lemma \ref{basic_ine}, we establish in the following lemma the existence and uniqueness of the mild solution of the linear equation $(\ref{NSE_linear})$.
\begin{lemma}\label{2.3*}
Suppose that Assumption \ref{A} holds. Let $f \in BC\left( {\mathbb{R},X} \right)$, $\psi \in (Z_1, Z_2)_{\theta, 1}$ and $\theta\in (0;\, 1)$ such that $1=(1-\theta)\alpha_1+\theta\alpha_2$. Then the equation $(\ref{NSE_linear})$ admits a unique mild solution $u$ satisfying
\begin{equation}\label{3.4}
{\left\| {u\left( t \right)} \right\|_Y} \leqslant \tilde{L}\norm{f}_{\infty,X},\ t \in \mathbb{R},
\end{equation}
for the  constant $\tilde{L}$ as in \eqref{L1es}.
\end{lemma}
\begin{proof}
In this proof and the rest of this paper, we use the weak (or weak$^*$) version of vector-valued integrals in the sense of the Pettis integral.
By assumption $(Z_1,Z_2)'_{\theta,1} = (Y_1,Y_2)_{\theta,\infty}$. We denote by $\left<.,. \right>$ the dual pair between $(Y_1,Y_2)_{\theta,\infty}$ and $(Z_1,Z_2)_{\theta,1}$. Then for each $\psi \in (Z_1,Z_2)_{\theta,1}$ we have the following estimates
\begin{eqnarray*}
\left|\left< \int_{-\infty}^t e^{-(t-\tau)A}Bf(\tau)d\tau,\psi \right> \right| &\leqslant& \int_{-\infty}^t \left|\left< e^{-(t-\tau)A}Bf(\tau),\psi \right>  \right|d\tau \cr
&=& \int_{-\infty}^t\left|\left< f(\tau), B'e^{-(t-\tau)A'}\psi \right>  \right|d\tau\cr
&\leqslant& \int_{-\infty}^t \norm{f(\tau)}_X \norm{B'e^{-(t-\tau)A'}\psi}_{X'}d\tau\cr
&\leqslant& \norm{f}_{\infty,X}\int_{-\infty}^t \norm{B'e^{-(t-\tau)A'}\psi}_{X'}d\tau\cr
&\leqslant& \tilde{L} \norm{f}_{\infty,X} \norm{\psi}_{(Z_1,Z_2)_{\theta,1}}.
\end{eqnarray*}
The last inequality holds due to Lemma \ref{L1es} and the transformation $\tau \mapsto t-\tau$.
\end{proof}
Lemma \ref{2.3*} implies that the solution operator $S:BC(\mathbb{R},X) \to BC(\mathbb{R},Y)$ defined by
\begin{equation}\label{3.8}
{S(f)(t)} : = \int_{ - \infty }^t {{e^{ - \left( {t - s} \right)A}}Bf(s)ds}, t \in\mathbb{R}
\end{equation}
is a bounded operator and $\norm{S} \leq \tilde{L}$ for the constant $\tilde{L}$ appearing in inequality \eqref{3.4}.
\begin{lemma}\label{soloperator}
Suppose that Assumption \ref{A} holds.  The following assertions hold
\begin{itemize}
\item[i)] If $f \in AP(\r,X)$, then $S(f)\in AP(\r,Y)$.
\item[ii)] If $f \in PAA_0(\r,X)$, then $S(f) \in PAA_0(\r,Y)$.
\end{itemize}
\end{lemma}
\begin{proof}
$i)$ Since $f$ is almost periodic, we obtain that for all $\varepsilon >0$ there exists a real number $l(\varepsilon) > 0$ such that for every $a \in \mathbb{R}$, we can find $T \in [a, a + l(\varepsilon)]$ such that
$$ \norm{f(t+T) - f(t)}_{X}<\varepsilon, \,\,\,\, t\in \mathbb{R}.$$ 
Then, we have
\begin{eqnarray*}
\norm{S(f)(t+T) - S(f)(t)}_Y &=& \norm{\int_{-\infty}^t e^{-(t-\tau)A}B[f(\tau+T) - f(\tau)]d\tau}_Y \cr
&=& \norm{\int_0^\infty e^{-\tau A}B[f(t-\tau+T) - f(t-\tau)]d\tau}_Y \cr 
&\leq& \tilde{L}\norm{f(.+T) - f(.)}_{\infty,X}\cr
&\leq& \varepsilon\tilde{L}.
\end{eqnarray*}
Therefore, the function $S(f)(t):=\int_{-\infty}^t e^{-(t-\tau)A}B f(\tau) d\tau$ belongs to  $AP(\mathbb{R},Y)$.\\

$ii)$ To prove this assertion we develop the method in \cite[Lemma 3.8]{Di2006} or \cite[Theorem 3.4]{Di2007} by replacing the condition for the semigroup $e^{-tA}$ from the exponential stability to  the polynomial stability.

We need to show that   
\begin{equation}
\lim_{r\rightarrow\infty}\frac{1}{2L} \int_{-L}^L \left\|\int_{-\infty}^t e^{-(s-\tau)A}B f(\tau) d\tau \right\|_{Y} dt= 0.
\end{equation}
This corresponds to prove that
\begin{equation}
\lim_{r\rightarrow\infty}\frac{1}{2L} \int_{-L}^L \sup_{\norm{\psi}\leq 1} \left| \int_{-\infty}^t \left< e^{-(s-\tau)A}B f(\tau),\psi \right> d\tau \right| dt= 0,
\end{equation}
for all $\psi \in (Z_1,Z_2)_{\theta,1}$ such that $\norm{\psi}:= \norm{\psi}_{(Z_1,Z_2)_{\theta,1}}\leq 1$.

Indeed, for each $\psi \in (Z_1,Z_2)_{\theta,1}$ with $\norm{\psi}:= \norm{\psi}_{(Z_1,Z_2)_{\theta,1}}\leq 1$ we have
\begin{eqnarray}\label{1'}
&&\dfrac{1}{2L}\int_{-L}^{L}\sup_{\norm{\psi}\leq 1}\left|\int_{-\infty}^t \left<e^{-(t-\tau)A}B f(\tau) ,\psi \right>d\tau \right| dt \cr
&=&\dfrac{1}{2L}\int_{-L}^{L}\sup_{\norm{\psi}\leq 1}\left|\int_{-\infty}^{-L} \left<e^{-(t-\tau)A}B f(\tau) ,\psi \right>d\tau \right| dt \cr
&&+\dfrac{1}{2L}\int_{-L}^{L}\sup_{\norm{\psi}\leq 1}\left|\int_{-L}^t \left<e^{-(t-\tau)A}B f(\tau) ,\psi \right>d\tau \right| dt. 
\end{eqnarray}
Recall that we proved inequality \eqref{L1es} in Lemma \ref{basic_ine} that
$$\int_0^\infty\norm{B'e^{-\tau A'}\psi}_{X'}d\tau \leqslant \tilde{L}\norm{\psi}_{(Z_1,Z_2)_{\theta,1}}$$
for all $\psi\in \left(Z_1,Z_2\right)_{\theta, 1}$.
By changing variable $\tau:=t-\tau$ we get
\begin{equation}\label{L1Ba}
\int_{-\infty}^t \left\|B'e^{-(t-\tau)A'} \psi \right\|_{X'}d\tau \leqslant \tilde{L}\norm{\psi}_{(Z_1,Z_2)_{\theta,1}}.
\end{equation}
This shows that for each $\varepsilon >0$, there exists $L_0(\varepsilon)\in \mathbb{R}$ depends only on $\varepsilon$ satisfying that for all $L > L_0(\varepsilon)$, we have
$$\int_{-\infty}^{-L} \norm{ B'e^{-(t-\tau)A'} \psi }_{X'} dt \leqslant \varepsilon\left\|\psi\right\|_{\left(Z_1,Z_2\right)_{\theta, 1}}$$
for all $\psi\in \left(Z_1,Z_2\right)_{\theta, 1}$.
Hence, for all $L>L_0(\varepsilon)$ we have
$$\sup_{\norm{\psi}\leqslant 1}\int_{-\infty}^{-L} \norm{ B'e^{-(t-\tau)A'} \psi }_{X'} dt \leqslant \sup_{\norm{\psi}\leqslant 1}\varepsilon\left\|\psi\right\|_{\left(Z_1,Z_2\right)_{\theta, 1}} \leqslant \varepsilon.$$
Therefore, we establish that
\begin{eqnarray}\label{2}
&&\dfrac{1}{2L}\int_{-L}^{L}\sup_{\norm{\psi}\leq 1}\left| \int_{-\infty}^{-L} \left<e^{-(t-\tau)A}B f(\tau) ,\psi \right> d\tau \right| dt \cr
&\leq&\dfrac{1}{2L}\int_{-L}^{L}\sup_{\norm{\psi}\leq 1}{\int_{-\infty}^{-L} \left|\left<e^{-(t-\tau)A}B f(\tau) ,\psi \right>\right|d\tau dt }\cr
&=& \dfrac{1}{2L}\int_{-L}^{L}\sup_{\norm{\psi}\leq 1}{\int_{-\infty}^{-L} \left|\left<f(\tau),B'e^{-(t-\tau)A'}\psi \right>\right|d\tau dt }\cr
&\leq& \dfrac{1}{2L}\int_{-L}^{L}\sup_{\norm{\psi}\leq 1}{\int_{-\infty}^{-L} \norm{f(\tau)}_X \norm{B'e^{-(t-\tau)A'}\psi}_{X'} d\tau dt }\cr
&\leq& \dfrac{\norm{f}_{\infty,X}}{2L}\int_{-L}^{L}\sup_{\norm{\psi}\leq 1}{\int_{-\infty}^{-L} \norm{B'e^{-(t-\tau)A'}\psi}_{X'} d\tau dt }\cr
&\leq& \dfrac{\norm{f}_{\infty,X}}{2L}\int_{-L}^{L} \varepsilon dt = \varepsilon \norm{f}_{\infty,X}.
\end{eqnarray}

On the other hand, there exists a sequence $(\psi_n)_{n\in \mathbb{N}}$ in $(Z_1,Z_2)_{\theta,1}$ such that $\norm{\psi_n}:= \norm{\psi_n}_{(Z_1,Z_2)_{\theta,1}}\leqslant 1$ and
$$\lim_{n\to \infty}\left| \int_{-L}^t \left< e^{-(t-\tau)A}Bf(\tau),\psi_n \right> d\tau  \right| = \sup_{\norm{\psi}\leqslant 1}\left| \int_{-L}^t \left< e^{-(t-\tau)A}Bf(\tau),\psi\right> d\tau  \right|.$$
We have
\begin{eqnarray*}
\left| \int_{-L}^t \left< e^{-(t-\tau)A}Bf(\tau),\psi_n \right> d\tau  \right| &\leqslant& \int_{-L}^t\norm{f(\tau)}_X\norm{B'e^{-(t-\tau)A'}\psi_n}_{X'}d\tau \cr
&\leqslant& \norm{f}_{\infty,X}\int_{-L}^t\norm{B'e^{-(t-\tau)A'}\psi_n}_{X'}d\tau \cr
&\leqslant& \norm{f}_{\infty,X}\int_{-\infty}^t\norm{B'e^{-(t-\tau)A'}\psi_n}_{X'}d\tau \cr
&\leqslant& \norm{f}_{\infty,X}\tilde{L}\norm{\psi_n}_{(Z_1,Z_2)_{\theta,1}} \hbox{   (due to \eqref{L1Ba})}\cr
&\leqslant& \tilde{L}\norm{f}_{\infty,X}.
\end{eqnarray*}
Since $\tilde{L}\norm{f}_{\infty,X}$ is integrable on $[-L, L]$, by using the dominated convergence theorem we have that $\sup_{\norm{\psi}\leqslant 1}\left| \int_{-L}^t \left< e^{-(t-\tau)A}Bf(\tau),\psi\right> d\tau  \right|$ is integrable on $[-L, L]$ and
\begin{eqnarray}\label{Domi}
&&\lim_{n\to\infty}\int_{-L}^L\left| \int_{-L}^t \left< e^{-(t-\tau)A}Bf(\tau),\psi_n\right> d\tau  \right|\cr
&=&\int_{-L}^L\sup_{\norm{\psi}\leqslant 1}\left| \int_{-L}^t \left< e^{-(t-\tau)A}Bf(\tau),\psi\right> d\tau  \right|.
\end{eqnarray}
Therefore, using inequalities \eqref{L1es} and \eqref{Domi} we have the following estimates
\begin{eqnarray*}
&&\frac{1}{2L} \int_{-L}^L \sup_{\norm{\psi}\leq 1}\left|\int_{-\infty}^t \left< e^{-(t-\tau)A}B f(\tau),\psi \right> d\tau \right| d t \cr
&=& \lim_{n\to\infty}\frac{1}{2L}\int_{-L}^L\left| \int_{-L}^t \left<f(\tau),B'e^{-(t-\tau)A'}\psi_n\right> d\tau  \right| dt \hbox{    (due to \eqref{Domi}) }\cr
&\leqslant& \lim_{n\to\infty}\frac{1}{2L}\int_{-L}^L \int_{-L}^t \left|\left< f(\tau), B'e^{-(t-\tau)A'}\psi_n\right> \right| d\tau dt\cr
&\leqslant& \lim_{n\to\infty}\frac{1}{2L}\int_{-L}^L \int_{-L}^t \norm{f(\tau)}_X \norm{B'e^{-(t-\tau)A'}\psi_n}_{X'} d\tau dt\cr
&=& \lim_{n\to\infty}\frac{1}{2L}\int_{-L}^L \int_0^{t+L} \norm{f(t-\xi)}_X \norm{B'e^{-\xi A'}\psi_n}_{X'} d\xi dt \hbox{    (where $\xi:=t-\tau$)}\cr
&=& \lim_{n\to\infty}\frac{1}{2L}\int_0^{2L} \int_{\xi-L}^L \norm{f(t-\xi)}_X \norm{B'e^{-\xi A'}\psi_n}_{X'} dt d\xi\cr
&=& \lim_{n\to\infty}\frac{1}{2L}\int_0^{2L} \int_{-L}^{L-\xi} \norm{f(s)}_X \norm{B'e^{-\xi A'}\psi_n}_{X'} ds d\xi  \hbox{    (where $s:=t-\xi$)}\cr
&\leqslant& \lim_{n\to\infty}\frac{1}{2L}\int_0^{2L} \int_{-L}^L \norm{f(s)}_X \norm{B'e^{-\xi A'}\psi_n}_{X'} ds d\xi\cr
&=& \lim_{n\to\infty}\frac{1}{2L} \int_0^{2L} \norm{B'e^{-\xi A'}\psi_n}_{X'}  \int_{-L}^L \norm{f(s)}_X ds d\xi\cr
&\leqslant& \lim_{n\to\infty}\frac{1}{2L} \int_0^\infty \norm{B'e^{-\xi A'}\psi_n}_{X'}d\xi  \int_{-L}^L \norm{f(s)}_X ds\cr
&\leqslant& \lim_{n\to\infty}\frac{\tilde{L}\norm{\psi_n}}{2L}\int_{-L}^L \norm{f(s)}_X ds \hbox{    (due to \eqref{L1es})}\cr
&\leq & \frac{\tilde{L}}{2L} \int_{-L}^L  \norm{f(s)}_{X} ds.  
\end{eqnarray*}
The property $f\in PAA_{0}(\mathbb{R}, X)$ leads to
$$\lim_{L\rightarrow\infty}\frac{1}{2L} \int_{-L}^L \left\| f(s) \right\|_{X}ds = 0.$$
Hence, there exists $L_1\in\mathbb{R}$ such that
$$\frac{1}{2L} \int_{-L}^L \left\| f(s) \right\|_{X}ds \leqslant \varepsilon,$$
for all $L> L_1$. Therefore, we have that
\begin{equation}\label{3}
\dfrac{1}{2L}\int_{-L}^{L}\sup_{\norm{\psi}\leq 1} \left|\int_{-L}^{t} \left<e^{-(t-\tau)A}B f(\tau) ,\psi \right>d\tau \right|dt \leq \tilde{L}\varepsilon,
\end{equation} 
for all $L> L_1$.
 
By combining \eqref{1'}, \eqref{2} and \eqref{3}, we get
$$\frac{1}{2L} \int_{-L}^L \sup_{\norm{\psi}\leq 1} \left|\int_{-\infty}^t \left< e^{-(t-\tau)A}B f(\tau),\psi \right> d\tau \right| d t \leq \varepsilon(\norm{f}_{\infty,X}+\tilde{L}),$$
for all $L \geq \max\left\{L_0, L_1\right\}$ and all $\psi \in (Z_1,Z_2)_{\theta,1}$ such that $\norm{\psi}\leq 1$. Hence
$$\frac{1}{2L} \int_{-L}^L \left\|\int_{-\infty}^t e^{-(t-\tau)A}B f(\tau) d\tau \right\|_Y d t \leq \varepsilon(\norm{f}_{\infty,X}+\tilde{L}).$$
Therefore, 
$$\lim_{L\rightarrow\infty}\frac{1}{2L} \int_{-L}^L \norm{S(f)(t)}_Y dt= 0.$$
\end{proof}

Using Lemma \ref{soloperator} we establish the existence and uniqueness of PAP-mild solution to the linear equation \eqref{NSE_linear} in the following theorem:
\begin{theorem}\label{thm1}
Suppose that Assumption \ref{A} holds. We have that if $f \in PAP(\r,X)$, then $S(f) \in PAP(\mathbb{R},Y)$. This means that the linear equation \eqref{NSE_linear} has a unique pseudo
almost periodic mild solution $\hat{u}\in PAP(\r,Y)$ for each inhomogeneous part $f\in PAP(\mathbb{R},X)$. Moreover, $\hat{u}$ satisfies
\begin{equation}\label{3.4'}
\norm{\hat{u}(t)}_Y \leqslant \tilde{L}\norm{f}_{\infty,X},\, t \in \mathbb{R},
\end{equation}
for the  constant $\tilde{L}$ as in \eqref{L1es}.
\end{theorem}
\begin{proof}
\noindent
Since $f\in PAP(\r,X)$, we put $f = g + \phi,$ where $g \in AP(\r,X)$ and $\phi\in PAA_0(\r,X)$. We have
\begin{eqnarray*}
S(f)(t) = \int_{-\infty}^t e^{-(t-\tau)A}B f(\tau)d\tau &=& \int_{-\infty}^t e^{-(t-\tau)A}B g(\tau) d\tau + \int_{-\infty}^t e^{-(t-\tau)A}B \phi(\tau) d\tau\cr
&=& S(g)(t) + S(\phi)(t).
\end{eqnarray*}
Using Lemma \ref{soloperator} we have that $S(g) \in AP(\r,Y)$ and $S(\phi)\in PAA_0(\r,Y)$. Therefore, $S(f) \in PAP(\r,Y)$ and Equation \eqref{NSE_linear} has a unique PAP-mild solution $\hat{u}$ such that \eqref{3.4'} by Lemma \ref{2.3*}. 
\end{proof}

\section{Solutions to semi-linear equations}\label{Section 3}
\subsection{The existence and uniqueness}
In this section we consider the semi-linear evolution equations
\begin{equation}\label{NSE_nonlinear}
u'\left( t \right) + Au\left( t \right) = BG\left( u \right)\left( t \right),t \in \mathbb{R},
\end{equation}
where the nonlinear part (also called the Nemytskii operator) $G$ maps from $BC(\mathbb{R},Y)$ into $BC(\mathbb{R},X)$ and $A$, $X$, $Y$ are the operator and the interpolation spaces similarly as defined in the linear equations \eqref{NSE_linear}.

A function $u\in C(\mathbb{R}, Y)$ is said to be a mild solution to $(\ref{NSE_nonlinear})$ if it satisfies the integral equation (see \cite[equation $(17)$]{HiHuySe}):
\begin{equation}\label{4.2}
u(t) = \int_{-\infty}^t {{e^{ -(t - \tau)A}}BG\left( u \right)\left( \tau  \right)d\tau. }
\end{equation}
To establish the existence and uniqueness of the PAP-mild solution for the semi-linear equation \eqref{NSE_nonlinear}, we need the following assumptions on the Nemytskii operator $G$:
\begin{assumption}\label{GAA}
We assume that the operator $G: BC\left( {\mathbb{R},Y} \right) \to BC\left( {\mathbb{R},X} \right)$ maps pseudo almost periodic functions to pseudo almost periodic functions and there is a positive constant $C$ such that the following estimate holds
$$\left\| G(u) - G(v) \right\|_{\infty,X} \leqslant C\left\| u - v \right\|_{\infty,Y}$$
for $u, v \in B(0,\rho) = \left\{ w \in BC(\mathbb{R},Y): \norm{w}_{\infty,Y} \leqslant \rho \right\}$.
\end{assumption}
The following theorem shows the existence and uniqueness of the $PAP$-mild solution to equation $(\ref{NSE_nonlinear})$ in a small ball of the Banach space $PAP(\mathbb{R},Y)$.
\begin{theorem}\label{thm2}
Suppose that $e^{-tA}B$ satisfies the polynomial estimates as in Assumption \ref{A} and Assumption $\ref{GAA}$ holds for Nemytskii operator $G$ with $C$ and $\norm{G(0)}_{\infty,X}$ being small enough. Then there exists a unique pseudo almost periodic mild solution $\hat u \in PAP(\r,Y)$ to the semi-linear equation $(\ref{NSE_nonlinear})$.
\end{theorem}
\begin{proof}
We denote by $B^{PAP}(0, \rho):=\{v\in PAP(\mathbb{R},Y)\mid \norm{v}_{\infty,Y}\leq \rho\}$ the ball centered at $0$ with radius $\rho >0$ in the space $PAP(\mathbb{R},Y)$.  

For each $v\in B^{PAP}(0,\rho)$ we consider the linear equation
\begin{equation}
u'(t)+Au(t) = BG(v)(t)
\end{equation}
Using Theorem \ref{thm1}, the above equation has a unique PAP-mild solution defined by
\begin{equation}
u(t)=\int_{-\infty}^t e^{-(t-\tau)A}BG(v)(\tau)d\tau,\, t\in \mathbb{R},
\end{equation}
and we can define the solution operator (see \eqref{3.8}) as
\begin{equation}
S(G(v))(t):=u(t).
\end{equation}
Then, we define the mapping $\Phi$ by
\begin{align*}
v &\longmapsto S\left( {G\left( v \right)} \right).
\end{align*} 
Next, we prove that $\Phi$ maps $B^{PAP}(0, \rho)$ into itself and is a contraction mapping on $B(0, \rho)$. Indeed, we  have
$$\Phi (v)(t) = S(G(v))(t) = \int_{-\infty}^t  e^{-(t-\tau)A}BG(v)(\tau) d\tau.$$
Since $G(v)$ belongs to $PAP(\mathbb{R},X)$ for $u\in PAP(\mathbb{R},Y)$, from Theorem \ref{thm1} it follows that $S(G(u))\in PAP(\mathbb{R},Y)$. Therefore, $\Phi$ maps $PAP(\mathbb{R},Y)$ into itself.
Moreover, applying Lemma \ref{2.3*} and using Assumption \ref{GAA} we can estimate that
\begin{eqnarray*}
\norm{\Phi (v)}_{\infty, Y} &\leqslant& \tilde{L} \norm{G(v)}_{\infty,X}=
\tilde{L} \norm{G(0)+G(v)-G(0)}_{\infty,X}\cr
&\leqslant& \tilde{L} \left( \norm{G(0)}_{\infty,X} + C\norm{v}_{\infty,Y} \right) \cr
&\leqslant& \tilde{L}(\norm{G(0)}_{\infty,X}+C\rho)<\rho,
\end{eqnarray*}
for $C$ and $\norm{G(0)}_{\infty,X}$ small enough. Therefore, $\Phi$ maps from $B^{PAP}(0,\rho)$ into itself.

We now show that $\Phi$ is a contraction mapping. Indeed, let $u$ and $v$ belong to the ball $B(0, \rho)$. Applying again Lemma \ref{2.3*} and Assumption \ref{GAA} we can estimate that
\begin{eqnarray*}
\left\| \Phi (u) - \Phi (v) \right\|_{\infty,Y} &\leqslant& \tilde L\left\| G(u) - G(v) \right\|_{\infty,X}\cr
&\leqslant& \tilde {L}C\left\| u - v \right\|_{\infty, Y},
\end{eqnarray*}
for all $u, v \in B^{PAP}(0, \rho)$. Hence, $\Phi$ is a contraction mapping for the small enough constant $C$. 

Then, the contraction principle yields the existence of a unique fixed point $\hat{u} \in B^{PAP}(0,\rho)$ of $\Phi$. By definition of $\Phi$ we have that $\hat u$ is PAP-mild solution of the semi-linear equation \eqref{4.2}.

In order to prove the uniqueness, we assume that $u, v \in PAP(\mathbb{R},Y)$ are two PAP-mild solutions to \eqref{4.2} such that $\norm{u}_{\infty,Y}\leqslant \rho$ and $\norm{v}_{\infty,Y}\leqslant \rho$. Then, using Lemma \ref{2.3*} and Assumption \ref{GAA} we have
\begin{eqnarray*}
\left\| {u - v} \right\|_{\infty,Y} &=& \left\| S\left( {G\left( u \right) - G\left( v \right)} \right) \right\|_{\infty,Y} \cr
&\leqslant& \tilde {L}\left\| {G\left( u \right) - G\left( v \right)} \right\|_{\infty,Y}\cr
&\leqslant& \tilde {L}C\left\| {u - v} \right\|_{\infty,Y}.
\end{eqnarray*}
Since $\tilde {L}C< 1$ for $C$ small enough,  we have the uniqueness of PAP-mild solution $\hat{u}$ in the ball $B^{PAP}(0,\rho)$.
\end{proof}

\subsection{Stability of the solutions}
In this section, we extend the results of polynomial stability of the mild solutions of the fluid dynamic equations obtained in \cite{HuyHaXuan} to the semi-linear equation \eqref{NSE_nonlinear} on the abstract interpolation spaces. In particular, we need some further abstract assumptions on the operator  $e^{-t A}B$ and the Nemytskii operator $G$ as in Assumption \ref{assum} below to establish the stability of the PAP-mild solution obtained in the previous section. 
\begin{assumption}\label{assum}
We assume that:
\begin{enumerate}
\item[i)] There exist positive numbers $0<\beta_1<1<\beta_2$ and Banach spaces $T,Q_1,Q_2,K_1,K_2$ such that $K_i$ is a pre-dual space of $Q_i$, $K_1\cap K_2$ is dense in $K_i$ for $i=1,2$, $e^{-tA}$ is a $C_0$-semigroup on $Q_1$ and $Q_2$ and we have the following polynomial estimates
\begin{equation}\label{1est}
\begin{gathered}
  \left\|e^{-t A}B\psi\right\|_{Q_1}\leqslant\tilde{M}t^{-\beta_1}\left\|\psi\right\|_T,\ t > 0, \hfill \\
  \left\|e^{-t A}B\psi\right\|_{Q_2}\leqslant\tilde{M}t^{-\beta_2}\left\|\psi\right\|_T,  \ t > 0, 
\end{gathered}
\end{equation}
for some constant $\tilde{M} > 0$ independent of $t$ and $\psi$.
\item[ii)] Putting $Q:= (Q_1, Q_2)_{\tilde{\theta}, \infty}$, where $0<\tilde{\theta}<1$ satisfied
$(1-\tilde{\theta})\beta_1 + \tilde{\theta}\beta_2=1$, there exist
positive constants $0 < \gamma < 1$ and $C_1 > 0$ such that 
\begin{equation}\label{2est}
\norm{e^{-tA}\psi}_Q \leq C_1 t^{-\gamma}\norm{\psi}_Y,\, t>0.
\end{equation}
\item[iii)] For the radius $\rho$ as in Assumption \ref{GAA} there exists $C_2 > 0$ such that the Nemytskii operator G satisfies
\begin{equation*}
\norm{G(v_1)-G(v_2)}_{\infty,T} \leq C_2 \norm{v_1-v_2}_{\infty,Q},
\end{equation*}
for all $v_1, \, v_2 \in B(0,\rho) \cap BC(\mathbb{R},Q) = \left\{ v \in BC(\r,Y) \cap BC(\mathbb{R},Q): \norm{v}_{\infty,Y} \leq \rho \right\}$.
\end{enumerate}
\end{assumption}
Now we extend \cite[Theorem 2.5]{HuyHaXuan} to state and prove the polynomial stability of the PAP-mild solution obtained in Theorem \ref{thm2} in the following theorem:
\begin{theorem}\label{Stabilitytheorem} 
Assume that $e^{-tA}B$ satisfies polynomial estimates as in Assumption \ref{A}, $G$ satisfies Assumption \ref{GAA} and $e^{-tA},\, B$, $G$ satisfy Assumption \ref{assum} with $C_2$ being small
enough. Let $\hat{u}\in PAP(\mathbb{R}, Y)$ be the pseudo almost periodic mild solution of \eqref{4.2} obtained in Theorem \ref{thm2}. Then, $\hat{u}$ is polynomial stable in the sense that: for any bounded mild solution $u \in BC(\mathbb{R}, Y)$ of the semi-linear equation \eqref{4.2}, if $\norm{\hat{u}}_{\infty,Y}$ and $\norm{u(0)-\hat{u}(0)}_Y$ are sufficiently small, then 
\begin{equation}\label{Stability_ine}
\left\|u(t)-\hat{u}(t)\right\|_Q \leqslant Dt^{-\gamma} \hbox{ for all }t>0,
\end{equation}
where $D$ is a positive constant independent of $u$ and $\hat{u}$.
\end{theorem}
\begin{proof}
For $t>0$ we can rewrite $u(t)$ and $\hat{u}(t)$  as follows:
$$u(t)=e^{-tA}u(0) + \int_0^t e^{-(t-\tau)A}BG(u)(\tau)d\tau,$$
$$\hat{u}(t)=e^{-tA}\hat{u}(0) + \int_0^t e^{-(t-\tau)A}BG(\hat{u})(\tau)d\tau,$$
where 
$$u(0) = \int_{-\infty}^0 e^{\tau A}BG(u)(\tau)d\tau \hbox{ and } \hat{u}(0) = \int_{-\infty}^0 e^{\tau A}BG(\hat{u})(\tau)d\tau.$$

By putting $v=u-\hat{u}$ we obtain that $v$ satisfies the integral equation
\begin{equation}\label{EqSta}
 v(t)= e^{-tA}(u(0) - \hat{u}(0)) + \int_0^{t} e^{-(t-\tau)A}B(G(u)(\tau)-G(\hat{u})(\tau))d\tau.
\end{equation} 
We set $\mathbf{M}:=\{v\in BC(\mathbb{R}_+, Y): \, \mathop {\sup }\limits_{t \in \mathbb{R}_+} {t^{\gamma}\left\|v(t)\right\|_Q}< \infty\}$
endowed with the norm 
$$\left\|v\right\|_\mathbf{M}: = \left\|v\right\|_{BC(\mathbb{R}_+, Y)} + \mathop {\sup }\limits_{t \in \mathbb{R}_+} {t^{\gamma}\left\|v(t)\right\|_Q}.$$

Putting $u(0):=u_0$ and $\hat u(0):=\hat u_0$ we now prove that if $\norm{\hat{u}}_{\infty,Y}$ (hence $\left\|\hat{u}\right\|_{BC(\mathbb{R}_+, Y)}$) and $\norm{u_0-\hat{u}_0}_Y$ are small enough, then equation \eqref{EqSta} has a unique solution on a small ball of $\mathbf{M}$. 
Indeed, for $v\in \mathbf{M}$, we consider the mapping
$$\Phi(v)(t):= e^{-tA}(u_0-\hat{u}_0) + \int_0^{t}e^{-(t-\tau)A}B\left(G(v+\hat{u})(\tau)-G(\hat{u})(\tau)\right)d\tau. $$
Setting ${\bf B}(0,\rho):= \{v\in \mathbf{M} : \left\|v\right\|_{\mathbf{M}}\leqslant \rho\}$ we  
 prove that for sufficiently small $\norm{\hat{u}}_{\infty,Y}$, $\norm{\hat{u}_0 - u_0}_Y$ and $C_2$, the map $\Phi$ acts from ${\bf B}(0,\rho)$ to itself and is a contraction mapping. Clearly, $\Phi(v)\in BC(\mathbb{R}_+, Y)$ for $v\in BC(\mathbb{R}_+, Y)$. Moreover
\begin{eqnarray*}
t^{\gamma}\Phi(v)(t) &=& t^{\gamma}e^{-tA}(u_0-\hat{u}_0)+ t^{\gamma}\int_0^t e^{-(t-\tau)A}B\left(G(v+\hat{u})(\tau)-G(\hat{u})(\tau)\right)d\tau\cr
&=& t^{\gamma}e^{-tA}(u_0-\hat{u}_0)+ t^{\gamma}\int_0^t e^{-\tau A}B\left(G(v(t-\tau)+\hat{u}(t-\tau))-G(\hat{u}(t-\tau))\right)d\tau\cr
&=& t^{\gamma}e^{-tA}(u_0-\hat{u}_0) + t^{\gamma}\int_0^t{F(\tau)d\tau}\,,
\end{eqnarray*}
where 
$$F(\tau):=e^{-\tau A}B\left(G(v(t-\tau)+\hat{u}(t-\tau))-G(\hat{u}(t-\tau))\right).$$
By inequality \eqref{2est} in Assumption \ref{assum} $ii)$ we have
\begin{equation}\label{Ine0}
\norm{t^{\gamma}e^{-tA}(u_0-\hat{u}_0)}_Q = t^{\gamma}\norm{e^{-tA}(u_0-\hat{u}_0)}_Q \leq C_1\norm{u_0 - \hat{u}_0}_Y.
\end{equation}

By Assumption \ref{assum} $i)$, we have $(K_1,K_2)'_{\tilde{\theta},1} = (Q_1,Q_2)_{\tilde{\theta},\infty}=Q$. We denote by $\left<.,.\right>$ the dual pair between $Q$ and $(K_1,K_2)'_{\tilde{\theta},1}$. The for each $\psi\in (K_1, K_2)_{\tilde{\theta}, 1}$ we have
\begin{eqnarray}
\left|\left<\int_{0}^t{F(\tau)d\tau}, \psi \right>\right| &\leqslant &\int_{0}^t\left|\left<F(\tau), \psi \right>\right| d\tau\cr
&\leqslant &\int_{0}^{t/2}\left|\left<F(\tau), \psi \right>\right| d\tau + \int_{t/2}^t\left|\left<F(\tau), \psi \right>\right| d\tau\,.\label{Eqi}
\end{eqnarray}
Since inequalities in \eqref{1est} in Assumption \ref{assum} $i)$, we have the following inequalities on the dual spaces
\begin{equation}\label{11est}
\begin{gathered}
  \left\|B'e^{-t A'}\psi\right\|_{T'}\leqslant\tilde{M}t^{-\beta_1}\left\|\psi\right\|_{Q_1'},\, t > 0, \hfill \\
  \left\|B'e^{-t A'}\psi\right\|_{T'}\leqslant\tilde{M}t^{-\beta_2}\left\|\psi\right\|_{Q_2'}, \, t > 0, 
\end{gathered}
\end{equation}
Therefore, by the same way as in the proof of Lemma \ref{inequalities_basic}, we can establish that
$$ \int_{0}^{\infty}\left\|B'e^{-\tau A'}\psi\right\|_{T'} d\tau\leqslant\tilde{N}\left\|\psi\right\|_{(K_1, K_2)_{\tilde{\theta}, 1}}.$$

Since $\norm{\hat{u}}_{\infty,Y}$ is small enough and $\norm{v}_{\infty,Y}<\rho$ we can consider $\norm{\hat{u}}_{\infty,Y}<\rho$ small enough such that $\norm{v+\hat{u}}_{\infty,Y}\leqslant \rho$. By using the Lipschitz property of $G$ in Assumption \ref{assum} $iii)$ the first integral in \eqref{Eqi} can be estimated as 
\begin{eqnarray}\label{F}
\int_{0}^{t/2}\left|\left<F(\tau), \psi \right>\right|d\tau&\leqslant&\int_{0}^{t/2}\left\|G(v(t-\tau)+\hat{u}(t-\tau))-G(\hat{u}(t-\tau)) \right\|_T \left\|B'e^{-\tau A'}\psi\right\|_{T'} d\tau\cr
&\leqslant&\int_{0}^{t/2}C_2\left\|v(t-\tau)\right\|_{Q}\left\|B'e^{-\tau A'}\psi\right\|_{T'} d\tau\cr
&\leqslant& \left(\frac{t}{2}\right)^{-\gamma}C_2\left\|v\right\|_{\bf M}\int_{0}^{\infty}\left\|B'e^{-\tau A'}\psi\right\|_{T'} d\tau\cr
&\leqslant&\tilde{N} 2^{\gamma} t^{-\gamma}C_2\left\|v\right\|_{\bf M}\left\|\psi\right\|_{(K_1, K_2)_{\tilde{\theta}, 1}}.
\end{eqnarray}

By \eqref{11est} and the fact that the canonical embedding $K_i \to Q_i' \,(i=1,2)$  is an isometry
we have that
\begin{eqnarray*}
\left\|B'e^{-\tau A'}\psi\right\|_{T'} &\leqslant& \tilde{M}\tau^{-\beta_1}\left\|\psi\right\|_{K_1},\cr
\left\|B'e^{-\tau A'}\psi\right\|_{T'} &\leqslant& \tilde{M}\tau^{-\beta_2}\left\|\psi\right\|_{K_2}
\end{eqnarray*}
for $\tau>0$. By applying Theorem \ref{GIT} with noting that $(1-\tilde{\theta})\beta_1 + \tilde{\theta}\beta_2=1$ and $(T',T')_{\tilde{\theta},1} = T'$ we obtain that
$$\norm{B'e^{-\tau A'}\psi}_{T'}<C' \tau^{-1}\norm{\psi}_{(K_1,K_2)_{\tilde{\theta},1}}$$
for some constant $C'>0$. 

Now the second integral in \eqref{Eqi} can be estimated as follows:
\begin{eqnarray}\label{S}
\int_{t/2}^t\left|\left<F(\tau), \psi \right>\right| d\tau &\leqslant&\int_{t/2}^t\left\|G(v(t-\tau)+\hat{u}(t-\tau))-G(\hat{u}(t-\tau))\right\|_{T}\left\|B'e^{-\tau A'}\psi\right\|_{T'} d\tau \cr
&\leqslant&C_2\left\|v\right\|_{\bf M}\int_{t/2}^t(t-\tau)^{-\gamma}\left\|B'e^{-\tau A'}\psi\right\|_{T'} d\tau\cr
&\leqslant& C'C_2\left\|v\right\|_{\bf M} \left( \int_{t/2}^t(t-\tau)^{-\gamma}\tau^{-1} d\tau \right)\left\|\psi\right\|_{(K_1, K_2)_{\tilde{\theta}, 1}}\cr
&\leqslant& C'C_2\left\|v\right\|_{\bf M} \frac{2}{t}\int_{t/2}^t(t-\tau)^{-\gamma} d\tau\left\|\psi\right\|_{(K_1, K_2)_{\tilde{\theta}, 1}} \cr
&\leqslant& \frac{C'2^\gamma }{1-\gamma}t^{-\gamma}C_2\left\|v\right\|_{\bf M}\left\|\psi\right\|_{(K_1, K_2)_{\tilde{\theta}, 1}}.
\end{eqnarray}
By combining inequalities \eqref{F} and \eqref{S}, we obtain that 
\begin{eqnarray}\label{Ine1}
&&\left\|\int_0^t e^{-\tau A}B(G(v(t-\tau)+\hat{u}(t-\tau))-G(\hat{u}(t-\tau)))d\tau\right\|_{Q}\cr
&\leqslant& \left( \frac{C'}{1-\gamma} + \tilde{N} \right)2^\gamma t^{-\gamma}C_2\left\|v\right\|_{\mathbf{M}},
\end{eqnarray}
for all $t>0$. 

Combining now \eqref{Ine0} and \eqref{Ine1} we obtain that 
$$ \left\|\Phi(v)\right\|_{\mathbf{M}}\leqslant C_1\norm{u_0-\hat{u}_0}_Y + D'C_2\left\|v\right\|_{\mathbf{M}}$$
for $D' =  \left( \frac{C'}{1-\gamma} + \tilde{N} \right)2^\gamma >0$. 
Therefore if $\norm{u_0-\hat{u}_0}_Y$, $\left\|\hat{u}\right\|_{\infty,Y}$
and $C_2$ are small enough, the mapping $\Phi$ acts the ball ${\bf B}(0,\rho)$ into itself.

Similarly as above, we have the following estimate
$$ \left\|\Phi(v_1)-\Phi(v_2)\right\|_{\mathbf{M}}
\leqslant 2D'C_2\left\|v_1-v_2\right\|_{\mathbf{M}}.$$
Therefore, $\Phi$ is a contraction mapping for sufficiently small 
$\norm{u_0-\hat{u}_0}_Y$, $\left\|\hat{u}\right\|_{\infty, Y}$ and $C_2$.
As the fixed point of $\Phi$, the function $v = u - \hat{u}$ belongs to $\mathbf{M}$. Inequality \eqref{Stability_ine} hence follows, and we obtain the stability of the small solution $\hat{u}$. The proof is completed.
\end{proof}

\section{Heat equations with rough coefficients}\label{Section 4}
In this section we will apply the abstract results obtained in the previous sections to the semi-linear diffusion equations with rough coefficients. Consider a measurable function $b : \mathbb{R}^d \to \mathbb{C}$ satisfying $b\in L^\infty(\mathbb{R}^d)$ and such that $\mathrm{Re}\,\, b \geqslant \delta > 0$ for some $\delta > 0$. Consider the semi-linear diffusion equations with rough coefficients
\begin{equation}\label{Rou1}
u'(t) -  b\Delta u(t) = g(t,u), \, \, \, (t,x)\in \mathbb{R}\times \mathbb{R}^d,
\end{equation}
where $d\geqslant 2$, $g(t,u) = |u(t)|^{m-1}u(t) + F(t)$ for some fixed constant $m \in \mathbb{N}, m\geqslant 2$ and a given bounded (on $\mathbb{R}$) function $F$. 

We know that the operator $-A$ defined on $L^p(\mathbb{R}^d)$ by $Au := -b\Delta u$ generates a bounded analytic semigroup (also called ultracontractive semigroup) $T(t):= e^{-tA}$ on $L^p(\mathbb{R}^d)$ for all $1<p<\infty$ (for more details see \cite[Section 7.3.2]{Are04}) such that 
$$(T(t)f)(x) = \int_{\mathbb{R}^d}K(t,x,y)f(y)dy, \, \, \, t>0 \hbox{  and a.e  } x,y \in \mathbb{R}^d,$$
where $K(t,x,y)$ is the heat kernel which verifies the following Gaussian estimate (see \cite[Section 7.4]{Are04}):
\begin{equation}\label{GaussEs}
|K(t,x,y)| \leq \frac{M}{t^{d/2}}e^{-\frac{a|x-y|^2}{bt}},\,\,\, x,y \in \mathbb{R}^d,
\end{equation}
for some constants $M,a,b>0$. The semi-linear equation \eqref{Rou1} can be rewritten as
\begin{equation}\label{Rou2}
u'(t) + Au(t) = g(t,u), \, \, \, (t,x)\in \mathbb{R}\times\mathbb{R}^d.
\end{equation}
The corresponding linear equation is
\begin{equation}\label{LRou}
u'(t) + Au(t) = F(t), \,\,\, (t,x) \in \mathbb{R}\times \mathbb{R}^d.
\end{equation}
The Gaussian estimate \eqref{GaussEs} of the heat kernel $K(t,x,y)$ allows us to verify the $L^p-L^q$ smoothing properties of the ultracontractive semigroup $e^{-tA}$ as follows:
\begin{equation}\label{API0}
\norm{e^{-tA}x}_{L^{q}(\mathbb{R}^d)} \leq Ct^{-\frac{d}{2}\left( \frac{1}{p} - \frac{1}{q} \right)}\norm{x}_{L^{p}(\mathbb{R}^d)} \hbox{  where  } 1<p\leq q <+\infty.
\end{equation}
We establish the $L^{p,r}-L^{q,r}$-smoothing estimates of $e^{-tA}$ in the following lemma.
\begin{lemma}
Let $1 < q < \infty$ and $1\leq r \leq +\infty$. Then, for $1 <p \leq q < +\infty$ the following inequality holds 
\begin{equation}\label{API}
\norm{e^{-tA}x}_{L^{q,r}(\mathbb{R}^d)} \leq Ct^{-\frac{d}{2}\left( \frac{1}{p} - \frac{1}{q} \right)}\norm{x}_{L^{p,r}(\mathbb{R}^d)}.
\end{equation}
\end{lemma}
\begin{proof}
For $1 < p < q$ there exist the numbers $1 < p_1 < p < p_2 < q$, $p<p_2 = q_1 < q< q_2$ and $0<\hat\theta<1$ such that 
$$\dfrac{1}{p} = \dfrac{1-\hat{\theta}}{p_1} + \dfrac{\hat{\theta}}{p_2} \hbox{  and  } \frac{1}{q} = \frac{1-\hat\theta}{q_1} + \frac{\hat\theta}{q_2}.$$
For example we can choose
$$\hat\theta= \frac{1}{2},\,\, p_1 = \frac{p(p+q)}{2q},\,\, p_2 = r_1 = \frac{p+q}{2},\,\,\, q_2 = \frac{q(p+q)}{2p}.$$
We have that
$$(L^{p_1}(\r^d),L^{p_2}(\r^d))_{\hat\theta,r} = L^{p,r}(\r^d) \hbox{   and   } (L^{q_1}(\r^d),L^{q_2}(\r^d))_{\hat\theta,r} = L^{q,r}.$$
Therefore, by using inequality \eqref{API0} and applying Theorem \ref{GIT} we obtain that
$$\norm{e^{-tA}x}_{L^{q,r}(\mathbb{R}^d)} \leq \left(C t^{-\frac{d}{2}\left( \frac{1}{p} - \frac{1}{q} \right)} \right)^{1-\hat\theta}\left(C t^{-\frac{d}{2}\left( \frac{1}{p} - \frac{1}{q} \right)} \right)^{\hat\theta} \norm{x}_{L^{p,r}(\mathbb{R}^d)} =C t^{-\frac{d}{2}\left( \frac{1}{p} - \frac{1}{q} \right)}\norm{x}_{L^{p,r}(\mathbb{R}^d)} .$$
Our proof is completed.
\end{proof}

As a consequence of Inequality \eqref{API} we have that $e^{-tA}B$, where $B=\mathrm{Id}$ (the identity operator) satisfies Assumption \ref{A} and the dual operators $B'=\mathrm{Id}$ and $e^{-tA'}$ verify Assumption \ref{assum}. Indeed, we choose the interpolation spaces $X,\, Y_1$ and $Y_2$ as follows:
$$X:=L^{\frac{d(m-1)}{2m},\infty}(\mathbb{R}^d) = L^{\frac{d(m-1)}{d(m-1)-2m},1}(\mathbb{R}^d)', \, \, Y_1:= L^{\frac{2d(m-1)}{5-m},\infty}(\mathbb{R}^d),\,\,\, Y_2:= L^{\frac{2d(m-1)}{m+3},\infty}(\mathbb{R}^d),$$
where $m,d$ are choosen such that
$$\frac{m}{4m-1} > d \geqslant 3 \hbox{  and  } 5 > m > \frac{d}{d-2}.$$
We choose $\theta =1/2$, then
$$(Y_1,Y_2)_{\frac{1}{2},\infty} = L^{\frac{d(m-1)}{2},\infty}(\mathbb{R}^d) = Y.$$
The preduals $Z_1$ and $Z_2$ of $Y_1$ and $Y_2$ are given by
$$Z_1 = L^{\frac{2d(m-1)}{(2d+1)(m-1)-4},1}(\mathbb{R}^d) \hbox{  and  } Z_2 = L^{\frac{2d(m-1)}{(2d-1)(m-1)-4},1}(\mathbb{R}^d)$$
respectively.

Using the inequality \eqref{API}, we have that
$$\norm{e^{-tA}\psi}_{Y_1} \leq L t^{-\frac{5}{4}}\norm{\psi}_X,$$
$$\norm{e^{-tA}\psi}_{Y_2} \leq L t^{-\frac{3}{4}}\norm{\psi}_X.$$
Hence, $e^{-tA}B = e^{-tA}$ satisfies the estimates in Assumption \ref{A} with $\alpha_1 = \frac{5}{4}$, $\alpha_2= \frac{3}{4}$ and $\theta = \frac{1}{2}$.

Now we have
$$\frac{2m}{d(m-1)} = \frac{2}{d(m-1)} + \frac{2j}{d(m-1)} + \frac{2(m-1-j)}{d(m-1)}.$$
Therefore, for all $u,\, v \in B(0,\rho) = \left\{ v\in BC(\mathbb{R},Y)| \norm{v}_{\infty,Y} \leq \rho \right\}$ and $t\in \mathbb{R}$ by using weak H\"older inequality we obtain that
\begin{eqnarray*}
&&\norm{|u(t)|^{m-1}u(t) - |v(t)|^{m-1}v(t)}_X \leq \sum_{j=0}^{m-1} \norm{|u(t)-v(t)||u(t)|^j|v(t)|^{m-1-j}}_X \cr
&\leq& \norm{|v(t)|^{m-1}}_{\frac{d}{2},\infty}\norm{u(t)-v(t)}_Y + \norm{|u(t)|^{m-1}}_{\frac{d}{2},\infty}\norm{u(t)-v(t)}_Y\cr
&+& \sum_{j=1}^{m-2}\norm{|u(t)|^j}_{\frac{d(m-1)}{2j},\infty} \norm{|v(t)|^{m-1-j}}_{\frac{d(m-1)}{2(m-1-j)},\infty}\norm{u(t)-v(t)}_Y\cr
&\leq& \norm{v(t)}^{m-1}_{\frac{d(m-1)}{2},\infty}\norm{u(t)-v(t)}_Y + \norm{u(t)}^{m-1}_{\frac{d(m-1)}{2},\infty} \norm{u(t)-v(t)}_Y\cr
&&+ \sum_{j=1}^{m-2}\norm{u(t)}^j_{\frac{d(m-1)}{2},\infty} \norm{v(t)}^{m-1-j}_{\frac{d(m-1)}{2},\infty}\norm{u(t)-v(t)}_Y\cr
&=&  \sum_{j=0}^{m-1}\norm{u(t)}^j_Y \norm{v(t)}^{m-1-j}_Y\norm{u(t)-v(t)}_Y\cr
&\leq& m\rho^{m-1} \norm{u(t)-v(t)}_Y,
\end{eqnarray*}
where $\norm{.}_{\frac{d(m-1)}{2j},\infty}$, $\norm{.}_{\frac{d(m-1)}{2(m-1-j)},\infty}$ are denoted the norms on $L^{\frac{d(m-1)}{2j},\infty}(\mathbb{R}^d)$ and $L^{\frac{d(m-1)}{2(m-1-j)},\infty}(\mathbb{R}^d)$ respectively.
Hence $g(t,u) = |u(t)|^{m-1}u(t) + F(t)$ satisfies the Lipschitz condition in Assumption \ref{GAA} with $C=m\rho^{m-1}$. The boundeness of $\norm{g(t,0)}_{\infty,X}$ holds due to the boundedness of $F$.

Now, fix any number $r> \frac{d(m-1)}{2}$. Since $1<\frac{r}{r-1}<\frac{dr}{d(r-1)-2r}$, we can choose real numbers $q_1$ and $q_2$ such that 
$$1<q_1< \frac{r}{r-1}<q_2 < \frac{dr}{d(r-1)-2r},$$
and there exists $\tilde{\theta}\in (0;\, 1)$ such that $\frac{r-1}{r} = \frac{1-\tilde{\theta}}{q_1} + \frac{\tilde{\theta}}{q_2}$. Therefore, we can determine the spaces $Q_1,\,Q_2,\,K_1,\, K_2,\, Q$ and $T$ in Assumption \ref{assum} $i)$ as follows:
$$T:= L^{\frac{dr}{d+2r},\infty}(\mathbb{R}^d) = L^{\frac{dr}{d(r-1)-2r},1}(\mathbb{R}^d)',$$
$$Q_1:= L^{\frac{q_1}{q_1-1},\infty}(\mathbb{R}^d),\, K_1 := L^{q_1,1}(\mathbb{R}^d),$$
$$Q_2:= L^{\frac{q_2}{q_2-1},\infty}(\mathbb{R}^d),\, K_2 := L^{q_2,1}(\mathbb{R}^d),$$
$$Q=(Q_1,Q_2)_{\tilde{\theta},\infty} = L^{r,\infty}(\mathbb{R}^d).$$
Using the inequality \eqref{API}, we have that
$$\norm{e^{-tA}\psi}_{Q_1} \leq \tilde{M} e^{-\beta_1}\norm{\psi}_T,$$
$$\norm{e^{-tA}\psi}_{Q_2} \leq \tilde{M} e^{-\beta_2}\norm{\psi}_T,$$
where $\beta_j\, (j=1,2)$ are chosen as 
$$\beta_j = \frac{d}{2} \left( \frac{1}{q_j} - \frac{d(r-1)-2r}{dr}  \right).$$
Therefore, we have that $0<\beta_2<1<\beta_1$ and $1 = (1-\tilde{\theta})\beta_1 + \tilde{\theta}\beta_2$.\\

We choose $\gamma = \frac{1}{m-1} - \frac{d}{2r}>0$ and by using again the inequality \eqref{API} we have
$$\norm{e^{-tA}\psi}_Q \leq C t^{-\gamma}\norm{\psi}_Y.$$
Therefore, $e^{-tA}B=e^{-tA}$ and $e^{-tA}$ satisfy the estimates in Assumption \ref{assum} $i)$ and $ii)$ respectively.

Lastly, we have
$$\frac{d+2r}{dr} = \frac{1}{r} + \frac{2j}{d(m-1)} + \frac{2(m-1-j)}{d(m-1)}.$$
Therefore, for all $u, \, v\in B(0,\rho) \cap BC(\mathbb{R},Q) = \left\{ v\in BC(\r,Y)\cap BC(\mathbb{R},Q)| \norm{v}_{\infty,Y} \leq \rho \right\}$ by using weak H\"older inequality we obtain that
\begin{eqnarray*}
&&\norm{|u(t)|^{m-1}u(t) - |v(t)|^{m-1}v(t)}_T \leq \sum_{j=0}^{m-1} \norm{|u(t)-v(t)||u(t)|^j|v(t)|^{m-1-j}}_T \cr
&\leq&\norm{|v(t)|^{m-1}}_{\frac{d}{2},\infty}\norm{u(t)-v(t)}_Q + \norm{|u(t)|^{m-1}}_{\frac{d}{2},\infty}\norm{u(t)-v(t)}_Q\cr
&&+ \sum_{j=1}^{m-2}\norm{|u(t)|^j}_{\frac{d(m-1)}{2j},\infty} \norm{|v(t)|^{m-1-j}}_{\frac{d(m-1)}{2(m-1-j)},\infty}\norm{u(t)-v(t)}_Q\cr
&\leq&\norm{|v(t)|^{m-1}}_{\frac{d}{2},\infty}\norm{u(t)-v(t)}_Q + \norm{|u(t)|^{m-1}}_{\frac{d}{2},\infty}\norm{u(t)-v(t)}_Q\cr
&&+ \sum_{j=1}^{m-2}\norm{u(t)}^j_{\frac{d(m-1)}{2},\infty} \norm{v(t)}^{m-1-j}_{\frac{d(m-1)}{2},\infty}\norm{u(t)-v(t)}_Q\cr
&=&  \sum_{j=0}^{m-1}\norm{u(t)}^j_Y \norm{v(t)}^{m-1-j}_Y\norm{u(t)-v(t)}_Q\cr
&\leq& m\rho^{m-1} \norm{u(t)-v(t)}_Q.
\end{eqnarray*}
Hence, $g(t,u) = |u(t)|^{m-1}u(t) + F(t)$ satisfies the Lipschitz condition for $G$ in Assumption \ref{assum} $iii)$ with $C_2=m\rho^{m-1}$.

Applying the abstract results obtained in Theorem \ref{thm1}, Theorem \ref{thm2} and Theorem \ref{Stabilitytheorem} in the previous sections, we obtain the existence, uniqueness and stability for the pseudo almost periodic mild solution of equations \eqref{Rou2} and \eqref{LRou} in the following theorem:
\begin{theorem}\label{app}
Let $X=L^{\frac{d(m-1)}{2m},\infty}(\mathbb{R}^d)$, $Y=L^{\frac{d(m-1)}{2},\infty}(\mathbb{R}^d)$ (where $d\geq 2, m\geqslant 2,\, m\in \mathbb{N}$) and $F \in PAP(\mathbb{R},Y)$, then the following assertions hold.
\begin{itemize}
\item[$(i)$] The linear equation \eqref{LRou} has a unique pseudo almost periodic mild solution $u\in PAP(\r,Y)$ such that 
$$\norm{u(t)}_{Y}\leq \tilde{L}\norm{F}_{\infty,X},\,\,\, t>0$$
where $\tilde{L}$ is a positive constant. 
 
\item[$(ii)$] If  $\norm{F}_{\infty,X}$ and $\rho>0$ are small enough, then the semi-linear equation \eqref{Rou2} has a unique pseudo almost periodic mild solution $\hat{u}$ in the small ball $B^{PAP}(0,\rho)$ of the Banach space $PAP(\mathbb{R},Y)$.

\item[$(iii)$] The  above solution $\hat u$ is polynomial stable in the sense that for any other solution $u \in BC(\mathbb{R},Y)$ of \eqref{Rou2}, if $\norm{\hat{u}}_{\infty,Y}$ and $\norm{u(0)-\hat{u}(0)}_Y$ are small enough, then we have
$$\norm{u(t)-\hat{u}(t)}_{L^{r,\infty}(\mathbb{R}^d)}  \leqslant \frac{C}{t^{\frac{1}{m-1}-\frac{d}{2r}}},\,\,\, t>0,$$
where $r > \frac{d(m-1)}{2}$.
\end{itemize}
\end{theorem}
\begin{remark}
Our abstract results can be also applied to the fluid dynamic equations as in \cite{GeHiHu,HuyHaXuan}.
\end{remark}
\noindent
\noindent

\end{document}